\documentclass[reqno]{amsart}
\usepackage{epsfig}
\usepackage{color}
\usepackage{amssymb,amsmath,amsthm,amstext,amsfonts}
\usepackage{psfrag}
\usepackage{url}
\usepackage{epstopdf}
\usepackage{epic,eepic}
\usepackage{upgreek}
\usepackage{amssymb}
\usepackage{mathrsfs}

\usepackage[colorlinks=true, linkcolor=blue, urlcolor=red, citecolor=blue]{hyperref}

\makeatletter \@addtoreset{equation}{section} \makeatother

\renewcommand\thetable{\thesection.\@arabic\c@table}

\theoremstyle{plain}
\newtheorem{maintheorem}{Theorem}
\newtheorem{maincorollary}{Corollary}
\newtheorem{mainproposition}{Proposition}

\newtheorem{proposition}{Proposition}[section]
\newtheorem{lemma}{Lemma}[section]

\newtheorem{remark}{Remark}[section]
\newtheorem{example}{Example}[section]

\newcommand{\htop}{h_{\topp}}

\newcommand{\Per}{\text{Per}}

\newcommand{\vep}{\varepsilon}

\newcommand{\diam}{\operatorname{diam}}
\newcommand{\dist}{\operatorname{dist}}

\newcommand{\topp}{\operatorname{top}}

\newcommand{\Ptop}{P_{\topp}}

\newcommand{\cP}{\mathcal{P}}

\newcommand{\cE}{\mathcal{E}}
\newcommand{\cM}{\mathcal{M}}

\newcommand{\cU}{\mathcal{U}}

\newcommand{\cA}{\mathcal{A}}

\newcommand{\Homeo}{\text{Homeo}(M)}

\newcounter{main}

\title[The centralizer of $C^r$-generic diffeomorphisms at hyperbolic basic sets is trivial]
{The centralizer of $C^r$-generic diffeomorphisms at hyperbolic basic sets is trivial}

\author[Jorge Rocha]{Jorge Rocha}
\address{Jorge Rocha, Departamento de Matem\'atica, Universidade do Porto, 
Rua do Campo Alegre, 687, 
4169-007 Porto, Portugal}
\email{jrocha@fc.up.pt}

\author[Paulo Varandas]{Paulo Varandas}
\address{Paulo Varandas, Departamento de Matem\'atica, Universidade Federal da Bahia\\
Av. Ademar de Barros s/n, 40170-110 Salvador, Brazil}
\email{paulo.varandas@ufba.br}

\begin{document}

\begin{abstract}
In the late nineties, Smale proposed a list of problems for the next century and, among these, it was conjectured 
that for every $r\ge 1$ a $C^r$-generic diffeomorphism has trivial centralizer. Our contribution here is to prove the
triviality of $C^r$-centralizers on hyperbolic basic sets. In particular, $C^r$-generic 
transitive Anosov diffeomorphisms have a trivial $C^1$-centralizer. 
These results follow from a more general criterium for expansive homeomorphisms with the gluing orbit property.
We also construct a linear Anosov diffeomorphism 
on $\mathbb T^3$ with discrete, non-trivial centralizer and with elements that are not roots. 
Finally, we prove that all elements in the centralizer of an Anosov diffeomorphism preserve some of its
maximal entropy measures, and use this to characterize the centralizer of linear Anosov diffeomorphisms on tori.

\end{abstract}

\date{\today}
\maketitle

\section{Introduction and statement of the main results}

\subsection*{Introduction}

In the late nineties Smale proposed a list of problems for the $21^{st}$ century and, among them, 
it is asked how typical are diffeomorphisms with trivial centralizer \cite{Sm}. 
The centralizer $\mathcal Z^r(f)$ 
of a $C^r$-diffeomorphism $f$, defined as the set of $C^r$-diffeomorphisms that commute with $f$, 
contains many information about possible symmetries of the dynamics. For instance, it may be used to 
determine when a diffeomorphism embeds as a time-$1$ map of a flow as done by Palis~\cite{Palis},
or to study the existence of smooth conjugacies for topologically conjugate circle diffeomorphisms,
a problem initiated by Herman in~\cite{Herman}. Thus, the problems and conjectures
established by Smale have a strong motivation from the interplay between dynamical systems and the algebraic 
properties of the linear models for hyperbolic automorphisms.
Important contributions to the conjecture have been given since the seventies. 
In \cite{Wa70}, Walters proved that the $C^0$-centralizer of expansive dynamics is discrete. 
Anderson~\cite{And} proved that for a $C^\infty$-open and dense set of Morse-Smale diffeomorphisms the centralizer
is also discrete.
Kopell \cite{Ko70} studied the triviality of the scentralizer of $C^r$ circle diffeomorphisms ($r\ge 2$) and 
linear transformations.
Palis and Yoccoz proved Smale's conjecture for Axiom A diffeomorphisms with the strong transversality
condition in the $C^\infty$ topology \cite{PY89,PY89b}.
In the real analytic setting, the first author proved that diffeomorphisms with trivial centralizer contain
a residual subset of $C^\omega$ Axiom A diffeomorphisms with the strong transversality condition \cite{Ro93}.
More recently, Bonatti, Crovisier and Wilkinson~\cite{BCW} 
proved that $C^1$-generic diffeomorphisms 
have trivial centralizer, providing  a solution of Smale's conjecture on the space of $C^1$-diffeomorphisms.
In that paper,  the authors introduced a notion of unbounded distortion, and show it is $C^1$-generic, 
and use it to prove that $C^1$-generic diffeomorphisms have trivial centralizer.
We also refer to \cite{JH91,BoRoVa} and references therein for some partial answers to Smale's conjecture in the context 
of flows hyperbolic and singular-hyperbolic flows.

The problem of the triviality of the centralizer in more regular topologies is still much an open question.
A classical strategy for obtaining discreteness for $C^r$-diffeomorphisms ($r>1$) is to use linearization 
at some hyperbolic periodic orbit and to reduce the problem of the centralizer back to the analysis of 
algebraic ingredients. The drawback of this strategy is that $C^r$-linearization is only guaranteed once
non-resonance conditions are satisfied by the eigenvalues of the derivative at that periodic point, and 
provided the diffeomorphism is sufficiently  smooth (the required  smoothness is given explicitly from the 
non-resonance conditions of the eigenvalues
and, for that reason, $C^{\infty}$-smoothness is assumed).
Indeed, there exists a $C^\infty$ open and dense subset of Anosov diffeomorphisms on tori and a $C^\infty$ open and dense subset of Axiom A diffeomorphisms with strong transversality which have trivial centralizer (see \cite{PY89b, PY89} respectively). 
The previous results were extended by the first author 
in the case of surfaces and $2\le r\le \infty$, there exists a $C^r$-open and dense subset of 
Axiom A diffeomorphisms with the no cycles condition whose elements have trivial centralizer~\cite{Fi08}.
Later, Fisher \cite{Fi09} proved that the elements of a $C^r$ ($2\le r\le \infty$) open and dense set of diffeomorphisms 
that exhibit a codimension-one hyperbolic and non-Anosov attractor have trivial centralizer on its basin of attraction.
Hence, the triviality of the centralizer at hyperbolic basic pieces is well known both in the case of $C^\infty$ diffeomorphisms and in the case of codimension one hyperbolic attractors.

Our main purpose in this paper is to contribute to description of the centralizer of $C^r$-diffeomorphisms at 
hyperbolic basic pieces ($r\ge 1$) removing both the $C^\infty$-smoothness and codimension one assumptions, 
establishing the triviality of the centralizer in the case of $C^r$-generic diffeomorphisms that
exhibit non-trivial hyperbolic basic pieces ($2\le r <\infty$). 
First, 
we prove a criterium for an element in the 
$C^0$-centralizer of expansive homeomorphisms with the periodic gluing orbit property to be a power of it (see Theorem~\ref{thm:PGT} below). Then, we prove that the assumptions in the criterium are satisfied by $C^1$ diffeomorphisms in the centralizer of $C^r$-generic diffeomorphisms and, consequently, 
$C^r$-generic hyperbolic basic pieces (including transitive Anosov diffeomorphisms) have trivial centralizer (cf. 
Corollaries~\ref{cor:0} and ~\ref{cor:A}).
As  Anosov diffeomorphisms may have non-trivial centralizers and may have positive entropy (even Anosov) diffeomorphisms in their centralizer (see Example~\ref{thm:C}) then it is important to provide a general
characterization for elements in the centralizer of Anosov diffeomorphisms. 
For that purpose, we prove in Theorem~\ref{thm:measures} that if the measure theoretical entropy function 
of a finite entropy homeomorphism is lower semi-continuous then all elements of the centralizer preserve some 
of its measures of maximal entropy. In consequence, the centralizer of Anosov automorphisms on tori 
is formed by volume preserving diffeomorphisms (Corollary~\ref{cor:Zcons}) and any $C^2$-partially hyperbolic 
diffeomorphism in its centralizer has positive entropy (cf. Corollary~\ref{cor:phm}).
Finally, we finish this article with some questions.
\subsection*{Preliminaries}

Let $f \in \text{Diff}^{\, r}(M)$, $r\ge 1$.
Let ${\text{Per}(f)}$ denote the set of periodic points for $f$ and $\Omega(f) \subset M$ denote the non-wandering set of $f$. An $f$-invariant set $\Lambda$ is \emph{transitive} if it has a dense orbit, that is, there exists
$x\in \Lambda$ so that $\overline{\mathcal O_f (x)} := \overline{ \{  f^n(x) : n\in \mathbb Z \} } = \Lambda$.
We say that a compact $f$-invariant set $\Lambda \subset M$ is a \emph{hyperbolic basic set} for $f$ if it is transitive
and admits a hyperbolic splitting: there is a $Df$-invariant splitting 
$T_{\Lambda} M = E^s \oplus E^u$ and constants $C>0$ and $\lambda\in (0,1)$ so that 
$
\| Df^n(x)\mid_{E^s_x} \| \le C \lambda^n 
	\quad\text{and}\quad
	\| (Df^n(x)\mid_{E^u_x})^{-1} \| \le C \lambda^n 
$
for every $x\in \Lambda$ and $n\ge 1$. 
A periodic point $p\in \text{Per}(f)$ is hyperbolic if its orbit is a hyperbolic set.

We say that $f \in \text{Diff}^{\, r}(M)$, $r\ge 1$, is \emph{Axiom A} if (i) $\overline{\text{Per}(f)}=\Omega(f)$
and (ii) $\Omega(f)$ is a uniformly hyperbolic set.
Clearly all periodic points of Axiom A diffeomorphisms are hyperbolic. We say that $f$ is an \emph{Anosov difffeomorphism} 
if the whole manifold $M$ is a hyperbolic set for $f$. 
One should 
mention that not all manifolds admit Anosov diffeomorphisms. 
It follows from the spectral decomposition theorem 
that for any Axiom A diffeomorphism $f$ there are finitely many pairwise disjoint 
hyperbolic basic sets $(\Lambda_i)_{i=1\dots k}$ so that 
$\Omega(f) = \Lambda_1 \cup \Lambda_2 \cup \dots \cup \Lambda_k$. 
We refer the reader e.g. to \cite{Shub} for more details. 

Given $f \in \text{Diff}^{\, r}(M)$, $r\in \mathbb N \cup\{\infty\}$ and $0\le k \le r$, the \emph{$C^k$-centralizer} for $f$ is the subgroup of $\text{Diff}^k(M)$ defined as 
$$
Z^k(f) = \{ g\in \text{Diff}^k(M) \colon g\circ f = f\circ g \},
$$
where $\text{Diff}^0(M)$ stands for the space $\Homeo$ of homeomorphisms on $M$. 
For every $1\le k \le n$, it is clear that $Z^k(f)$ is a subgroup of $(\text{Diff}^k(M), \circ)$
which always contains the subgroup $\{ f^n \colon n\in\mathbb Z \}$.
Clearly, 
$
\mathcal Z^0(f) \supset \mathcal  Z^1(f) \supset \mathcal  Z^2(f) \supset \dots \supset \mathcal  Z^r(f) \supset \{ f^n \colon n\in\mathbb Z \},
$
and we say that $f \in \text{Diff}^{\, r}(M)$ has \emph{trivial $C^k$-centralizer} if 
$Z^k(f)=\{ f^n \colon n\in\mathbb Z \}.$
For simplicity, will say that $f \in \text{Diff}^{\, r}(M)$ has \emph{trivial centralizer} if its $C^r$-centralizer is trivial. 
Note the $C^0$-triviality of the $C^0$-centralizer of $f\in \text{Diff}^{\,r}(M)$, $r\ge 1$, implies on the triviality
of $\mathcal Z^k(f)$ for all $1\le k \le r$.

We now recall some ingredients from the thermodynamic formalism. Let $\cM(M)$ denote the space of Borelian
probability measures on $M$, endowed with the weak$^*$ topology. The space of $f$-invariant probability measures 
$\cM_f(M)\subset \cM(M)$ is a compact subset, and for any $\nu\in \cM_f(M)$ the \emph{entropy of $\nu$} is defined as 
$
h_\nu(f)=\sup \{h_\nu(f,\cP) \colon \cP \,\text{is a finite partition in}\, M\},
$ 
where 
$$
h_\nu(f,\cP)=\inf_{n\ge 1} \frac1n H(\bigvee_{j=0}^{n-1} f^{-j}(\cP))
\quad\text{and}\quad
H(\cP)=\sum_{P\in \cP} -\mu(P) \log \mu(P).
$$
An $f$-invariant probability measure $\mu$ is called an 
\emph{equilibrium state} for $f$ with respect to a continuous potential $\phi: M \to \mathbb R$
if it attains the supremum in the variational principle for the topological pressure 
\begin{equation}\label{eq:VP}
\Ptop(f,\phi)=\sup_{\nu\in \cM_f(M)} \Big\{h_\nu(f) +\int \varphi\, d\nu\Big\}
\end{equation}
In the case that $\phi\equiv 0$ the previous expression becomes the variational principle for the topological 
entropy and any measure that attains the supremum is called a \emph{maximal entropy measure}.
In the case that $\Lambda$ is a hyperbolic basic set for $f$ it is well known that for every H\"older continuous
potential $\phi$ there exists a unique equilibrium state for $f\mid_\Lambda$ with respect to $\phi$.
We refer the reader to Bowen's monograph~\cite{Bo75} for more details.

\subsection*{Statement of the main results}\label{sec:statements}

This section is devoted to the statement of our main results. Our first main result is inspired by ~\cite{JH91} and 
provides a criterium for the triviality of the $C^0$-centralizer for expansive homeomorphisms on compact metric spaces. 
We need two preliminary notions that we now recall. Recall that an homeomorphism $f$ is \emph{expansive} if  
there exists $\vep>0$ so that for any $x,y\in M$ there exists $n\in\mathbb Z$ such that $d(f^n(x), f^n(y))>\vep$.
Any such constant $\vep$ is called an expansivity constant for $f$.

Given a compact metric space $\Lambda$ and an homeomorphism $f \in \text{Homeo}(\Lambda)$, 
we say that $f$ has the {\em periodic gluing orbit property} if for any $\vep>0$ there exists $K=K(\vep) >0$ such that
for any points $x_0,x_1, \dots, x_k \in \Lambda$ and positive integers $n_0, n_1, \dots, n_k \ge 1$
there are positive integers $p_0,p_1, \dots, p_{k-1}, p_k  \le K(\vep)$ and $x\in \Lambda$
so that $d( f^k(x), f^k(x_0) ) <\vep$ for all $ 1\le k \le n_0$, that
 $d( f^{k+\sum_{j=0}^{i-1} (p_j+ n_j)} (x) , f^k(x_i) ) <\vep$ for all $1\le k \le n_i$ and $1 \le i \le k$, and 
$f^{\sum_{j=0}^{k} (p_j+ n_j)} (x)=x$.
This notion is weaker than specification
and holds e.g. for minimal isometries on tori and uniformly hyperbolic dynamics. 
In fact, the periodic gluing orbit property implies on a strong transitivity and on the denseness of
periodic orbits. Moreover, this property holds for homeomorphisms with the periodic shadowing 
property on each chain recurrence class of the non-wandering set. 
We refer the reader to \cite{BTV,BoVa, BoToVa,ST} for more details. 

Finally, given homeomorphisms $f, h: M \to M$, we say that $h$ \emph{preserves the periodic orbits of $f$}
if $h(\mathcal O_f(x))=\mathcal O_f(x)$ for every $x\in \Per(f)$. In other words, for each $x\in \Per(f)$ there exists 
$n(x)\in \mathbb Z$ so that $h(x)=f^{n(x)}(x)$. We are now in a position to state our first main
result.

\begin{maintheorem}\label{thm:PGT}
Let $\Lambda$ be a compact metric space and let $f : \Lambda \to \Lambda$ be an expansive homeomorphism. 
Assume 
that $f$ satisfies the periodic gluing orbit property. 
If $h\in \mathcal Z^0(f)$ preserves periodic orbits of $f$ then $h=f^k$ for some $k\in \mathbb Z$.
\end{maintheorem}

In what follows we deduce some consequences. The first one, which is a consequence of Theorem~\ref{thm:PGT} together with Lemma~\ref{le:periodic-spectra}, shows that the centralizer at 
hyperbolic basic pieces is typically trivial.

\begin{maincorollary}\label{cor:0}
Let $1\le r\le \infty$, $f_0\in \text{Diff}^{\,r}(M)$ and $\Lambda_{f_0} \subset M$ be a maximal invariant set for $f_0$. 
Let $\mathcal U \subset \text{Diff}^{\,r}(M)$ be an open neighborhood of $f_0$ and let
$U \subset M$ be open set such that the analytic continuation of the hyperbolic basic sets defined by
\begin{equation}\label{eq:analytic}
\mathcal U \ni f\mapsto \Lambda_f := \bigcap_{n\in \mathbb Z} f^n(U)
\end{equation}
is well defined.
There exists a $C^r$-open 
neighborhood $\mathcal U$ of $f$ and a residual subset $\mathcal R \subset \mathcal U$
so that $\mathcal Z^1(g\mid_{\Lambda_g})$ 
is trivial for every $g\in \mathcal R$.
\end{maincorollary}

It follows from the previous corollary that $C^r$-generic Axiom A diffeomorphisms, 
restricted to their non-wandering set, have trivial centralizer. 
One should also refer that the first author constructed an open set of transitive Anosov diffeomorphisms on 
$\mathbb T^2$ with a non-trivial $C^0$-centralizer~\cite{Ro05}, and consequently 
Corollary~\ref{cor:0} is no longer true when considering the $C^0$-centralizer.
Thus, there exists an element in the centralizer that permutes periodic orbits (of the same period).  
In what follows we are also interested in the $C^r$-centralizer of Anosov diffeomorphisms, $r\ge 1$.
By structural stability, the set $\cA^r(M)$ of transitive Anosov diffeomorphisms
is an open set in $\text{Diff}^{\, r}(M)$, $r\ge 1$. 
Among the classes of transitive diffeomorphisms we should refer the set $\cA^r_m(M)$ 
of volume preserving 
Anosov diffeomorphisms.

\begin{maincorollary}\label{cor:A}
Let $M$ be a compact Riemannian manifold that supports Anosov diffeomorphisms and $1\le r\le \infty$. 
There are $C^r$-residual subset $\mathcal R_1 \subset \cA^r(M)$ and $\mathcal R_2 \subset \cA_m^r(M)$ of Anosov diffeomorphisms so that the $C^1$-centralizer of every 
$f\in \mathcal R_i$ is trivial, $i=1,2$.
\end{maincorollary}

We observe that the previous result implies on explosion of differentiability. Indeed, given $1\le r \le \infty$
any $C^1$ diffeomorphism that commutes with a $C^r$-generic Anosov diffeomorphism is itself $C^r$-smooth.
As previously mentioned, whenever $r=1$ and $r=\infty$ the previous result is a direct consequence of \cite{BCW} and \cite{PY89}, respectively.
For intermediate regularity $1<r<\infty$, in order to use the criterium established by 
Theorem~\ref{thm:PGT} we avoid a countable number of conditions (each of which determines 
a closed set with empty interior in $\text{Diff}^{\,r}(M)$).
We now observe that the centralizer of a linear Anosov diffeomorphism may have rationally independent
Anosov diffeomorphisms. More precisely:

\begin{example}\label{thm:C}
There exists a linear Anosov automorphism $f$ on $\mathbb T^3$ whose $C^\infty$-centralizer is discrete,
is non-trivial and contains an Anosov diffeomorphism $h$ that does not satisfy any of the equations $h^n=f^m$
for $m,n\in \mathbb Z \setminus\{0\}$. 
In particular the Anosov diffeomorphism, whose construction is described in Section~\ref{sec:example}, 
does not belong to the residual subsets described in Corollary~\ref{cor:0}. 
\end{example}

\medskip

The main purpose now is to describe the elements in the centralizer of \emph{every} Anosov diffeomorphism on tori. 
In what follows we discuss the entropy  of common
invariant measures.

\begin{maintheorem}\label{thm:measures}
Let $M$ be a compact metric space and assume that $f \in \Homeo$ have finite topological entropy.
If the entropy map $h: \mathcal M_f(M) \to \mathbb R_+$ given by $\mu \mapsto h_\mu(f)$ is upper semicontinuous
then every $g\in \mathcal Z^0(f)$ preserves 
a maximal entropy measure of $f$. In addition, if $\mathcal Z^0(f)$ is finitely generated
then there exists of a maximal entropy measure that is preserved by \emph{all} elements in $\mathcal Z^0(f)$.
\end{maintheorem}

It is well known that the measure theoretical entropy function of expansive homeomorphisms, including Anosov diffeomorphisms, is upper semicontinuous (see e.g. \cite{Wa}).  
Thus the following is a direct consequence of Theorem~\ref{thm:measures}:

\begin{maincorollary}\label{cor:mme}
Assume that $f\in \text{Diff}^{\, 1}(M)$ is a transitive Anosov diffeomorphism. Then
every $g \in Z^0(f)$ preserves all the equilibrium states of $f$ associated to H\"older continuous potentials.
In particular, the unique maximal entropy measure for $f$ is preserved by all $g\in \mathcal Z^0(f)$
and, consequently,
$h_{\text{top}}(f) = \sup_{\mu \in \cM_f(M) \cap \cM_1(g)}  h_\mu(f).$
\end{maincorollary}

The previous result should be compared with \cite{S96}: 
if $f, g \in \text{Homeo}(M)$ are commuting  homeomorphisms, $h_{\text{top}}(f)>0$ and 
$g$ is expansive then $h_{\text{top}}(g)>0$. 
In the case of linear Anosov automorphisms the maximal entropy measure coincides with the Lebesgue measure.
Therefore, we deduce that all diffeomorphisms commuting with a 
volume preserving Anosov diffeomorphisms are themselves volume preserving. More precisely:

\begin{maincorollary}\label{cor:Zcons}
Assume that $f=f_A$ is a linear Anosov automorphism on $\mathbb T^n$. Then, for every $r\ge 1$, 
$$
\mathcal Z^r_m(f) := \mathcal Z^r(f) \cap \text{Diff}^{\,r}_m(M) = \mathcal Z^r (f).
$$ 
\end{maincorollary}

Every Anosov diffeomorphism $f$ on $\mathbb T^n$ is topologically conjugated to
a linear Anosov automorphism $f_A$ induced by a hyperbolic matrix $A\in GL(n,\mathbb Z)$ with
$|\det A|=1$ (cf. \cite{Man74}). Clearly $f_A$ is volume preserving and the Lebesgue measure is 
the unique maximal entropy measure.
Moreover, it is not hard to check that the $C^0$-centralizer for $f$ is homeomorphic to the centralizer of 
$f_A$ (see e.g. \cite[Theorem~2]{RV1}). 
Now, we relate our results with the ones due to Katok~\cite{Katok} on the centralizer of diffeomorphisms preserving an invariant measure. 
For every $n\ge 2$ there exists a unique full supported maximal entropy measure for an Anosov diffeomorphism $f$ on $\mathbb T^n$. Hence, the following is a direct consequence of Corollary 3.1 and Theorem 4.1 in \cite{Katok}:

\begin{maincorollary}\label{cor:Katok}
Let $f: \mathbb T^n \to \mathbb T^n$ be a linear Anosov diffeomorphism, $n\ge 2$. The following hold:
\begin{enumerate}
\item if $n=2$ and $g\in \mathcal Z^1(f)$ then there are $k,\ell \in \mathbb Z$ so that $f^k g^\ell=id$; and 
\item if $n=3$ and $g,h\in \mathcal Z^1(f)$ then there are $k,\ell, m \in \mathbb Z$ so that $f^k g^\ell h^m=id$.
\end{enumerate}
\end{maincorollary}

Relations of the form $f^k g^\ell h^m=id$ are often associated to the existence of \emph{roots} on the centralizer 
(e.g. if $g^2=f$ then $g$ is a root of $f$, and if $h^5=id$ then $h$ is a root of the identity) 
and to remove them constitutes an important step for establishing trivial centralizer 
(see for instance \cite{PY89}). 
Example~\ref{thm:C} above implies that Corollary~\ref{cor:Katok} (1) above is no longer true for Anosov diffeomorphisms on $\mathbb T^3$. It also implies that considering two elements in the centralizer at 
Corollary~\ref{cor:Katok} (2) is optimal: 
there exist commuting Anosov $C^\infty$ diffeomorphisms $f$ and $g$ such that $f^k g^\ell\neq id$ for all $k,\ell \in \mathbb Z$.

\medskip

Finally, we discuss some rigidity phenomenon relating elements in the centralizer with their topological entropy.
First we observe the following:

\begin{mainproposition}\label{prop:rigiditydim2}
If $f_A : \mathbb T^2  \to \mathbb T^2$ is a linear Anosov diffeomorphism and $g\in \mathcal Z^0(f_A)$ has positive entropy then $g$ is a root of an Anosov diffeomorphism. 
\end{mainproposition}

Our next result concerns the entropy of partially hyperbolic diffeomorphisms in the centralizer of an Anosov
diffeomorphism. 
We say that $g\in \text{Diff}^{\, 1}(M)$ is \emph{partially hyperbolic} if there exists a $Dg$-invariant splitting
$TM=E^u \oplus E^c$ and constants $C>0$ and $\lambda\in (0,1)$ so that $\|(Df^n(x)\mid_{E^u_x})^{-1}\|\le C \lambda^n$ and $\| (Df^n(x)\mid_{E^u_x})^{-1} \| \, \| (Df^n(x)\mid_{E^c_x} \| \le C\lambda^n$ for every $x\in M$ and every $n\ge 1$.
While Corollary~\ref{cor:mme} asserts that measures of maximal entropy are preserved by 
any element in the centralizer of a transitive Anosov diffeomorphism, it is not clear how to compute their entropy.
The following result shows that all partial hyperbolic diffeomorphisms commuting with an Anosov diffeomorphism
have positive topological entropy. More precisely:

\begin{maincorollary}\label{cor:phm}
If $1\le r \le \infty$ and $f\in \text{Diff}^{\, 1}(\mathbb T^n)$ is a linear Anosov diffeomorphism then every $g\in \mathcal Z^2(f)$
is volume preserving and 
$
\htop(g) \ge \int \sum_i \lambda_i(g,x)^+ \, dLeb(x) 
$
where $\sum_i \lambda_i(g,\cdot)^+$ denotes the sum of positive Lyapunov exponents of $g$
with respect to Lebesgue. In particular, if $g\in \mathcal Z^2(f)$ is partially hyperbolic then $\htop(g)>0$.
\end{maincorollary}

\section{$C^0$-trivial centralizers}\label{sec:proofs}

In what follows we will prove Theorem~\ref{thm:PGT}, whose proof is inspired by \cite{JH91}.
Let $\Lambda$ be a compact metric space and let $f : \Lambda \to \Lambda$ be an expansive 
homeomorphism satisfying the periodic gluing orbit property and let $h\in \mathcal Z^0(f)$ be fixed. 
Assume throughout that $h\in \mathcal Z^0(f)$ preserves the periodic orbits of $f$.

\begin{lemma}\label{le:equidistributed}
For every $\vep>0$ there exists a periodic point $q\in \Per(f)$ whose orbit is $\vep$-dense in $\Lambda$.
In particular $\overline{\Per(f)}=\Lambda$.
\end{lemma}

\begin{proof}
The proof is a direct consequence of the periodic gluing orbit property.
Indeed, by compactness of $\Lambda$, for every 
$\vep>0$ there exist a periodic point that is $\vep$-dense. Indeed, given $\vep>0$ pick any finite set
$\{x_1, x_2, \dots, x_k\} \subset \Lambda$ that is $\vep/2$-dense in $\Lambda$ and take 
$n_i=1$ for every $1\le i \le k$. The periodic gluing orbit property assures the existence of a positive integer
$K=K(\frac{\vep}{2})\ge 1$ of positive integers $p_1, \dots, p_k \le K$ and of a periodic point $q$, 
of period at most $(1+K)k$, such that
so that $d( q, x_1 ) <\vep$, that
 $d( f^{\sum_{j=0}^{i-1} (p_j+ 1)} (q) , x_i ) <\vep$ for all $2 \le i \le k$, and 
$f^{\sum_{j=0}^{k} (p_j+ 1)} (q)=q$.
In consequence $\mathcal O_f(q)$ is $\vep$-dense in $\Lambda$. Since $\vep>0$ was chosen arbitrary this also
proves the second claim in the lemma.
\end{proof}

 We proceed to prove that $h=f^k$ for some $k\in \mathbb Z$. 
As $h$ preserves periodic orbits, given $p\in \text{Per}(f)$ there exists a unique
$n(p) \in \mathbb Z \cap ]-\frac{\pi(p)}2, \frac{\pi(p)}2]$  so that $h(p)=f^{n(p)}(p)$. 
Throughout, let $n: \Per(f) \to\mathbb Z$ be as defined above  
and note that $n(q)=n(f^j(q))$ for every $q\in \Per (f)$ and $0\le j \le \pi(q)$.
We first prove that if $n(\cdot)$ is bounded then the theorem follows.

\begin{lemma}\label{le:power}
If $n(\cdot) : \Per(f) \to \mathbb Z$ is bounded then there exists $k\in \mathbb Z$ so that $h=f^k$.
\end{lemma}

\begin{proof}
Assume there exists $N_0\ge 1$ so that $|n(p)| \le N_0$  for every $p\in \Per(f)$. Then it makes sense to  consider the decomposition
$\text{Per}(f) = \bigsqcup_{|j|\le N_0} \big\{ p\in \text{Per}(f\mid_\Lambda) \colon  n(p)=j \big\}.$
By Lemma~\ref{le:equidistributed}, for every $\ell\ge 1$ there exists 
a periodic point $p_\ell \in \text{Per}(f)$ that is $\frac1\ell$-dense in $\Lambda$. 
Using the previous decomposition on the space of periodic points 
and the pigeonhole principle, there exists $k \in \{-N_0, \dots, N_0\}$ so that the set
$P_{k}:=\big\{ p\in \text{Per}(f) \colon  n(p)=k \big\}$ contains infinitely many periodic points of the family $(p_\ell)_{\ell\ge 1}$.
Hence $P_{k}$ is dense in $\Lambda$ and $h\mid_{P_{k}}= f^{k}$. This implies that $h=f^k$ and proves the lemma.
\end{proof}

The remaining of the proof is to assure the hypothesis of the previous lemma is satisfied. 
We need the following estimate on $n(\cdot)$.

\begin{lemma}\label{le:continuity}
Given $p\in \text{Per}(f)$ of prime period $\pi(p)\ge 1$ there exists $\eta=\eta_p>0$ (depending on $p$, $f$ and $h$) 
so that for every $q\in \text{Per(f)} \cap B(p,\eta)$ either $n(q)=n(p)$ or $|n(q)|> \frac{\pi(p)}2$.
\end{lemma}

\begin{proof}
Given $p\in \text{Per}(f)$ pick $\zeta>0$ small enough such that the collection of balls $\{B(f^j(p), \zeta)\}_{j=0\dots \pi(p)-1}$ is pairwise disjoint and 
\begin{equation}\label{eq:choiceBalls}
f^k(B(f^j(p), \zeta)) \cap \Big( \bigcup_{s\neq j+k} B(f^{s}(p), \zeta) \Big) = \emptyset
\end{equation}
for every $j, k \in \{0\dots \pi(p)-1\}$.
Clearly, the homeomorphism $\tilde h := h\circ f^{-n(p)}$ belongs to $\mathcal Z^0(f)$ and $\tilde h(p)=p$. Using
that $\tilde h$ is uniformly continuous, there exists $0<\eta<\zeta/2$ so that $d(q,p)<\eta$ implies 
$
d(\tilde h (q), \tilde h(p)) = d(f^{n(q)-n(p)} (q),p)<\frac{\zeta}2.
$
Since $q$ and $\tilde h(q)=f^{n(q)-n(p)} (q)$ belong to $B(p,\eta)$ and $h$ preserves the periodic orbits of $f$ 
(hence the same holds for $\tilde h$) we conclude that either $\tilde h(q)=q$ or $\tilde h(q) \in \mathcal O_f(q) 
\setminus\{q\}$. These correspond to the cases that $n(q)=n(p)$ or, using \eqref{eq:choiceBalls}, 
that $|n(q)-n(p)| \ge \pi(p)$, respectively. This proves the dichotomy that $n(q)=n(p)$ or $|n(q)| \ge \pi(p)/2$
as claimed in the lemma.
\end{proof}

Theorem~\ref{thm:PGT} will follow as a consequence of Lemma~\ref{le:power} together with the following proposition.

\begin{proposition}
$n(\cdot) : \Per(f) \to \mathbb Z$ is bounded.
\end{proposition}

\begin{proof}
Assume by contradiction that $n(\cdot) : \Per(f) \to \mathbb Z$ is unbounded.
As $f$ is expansive let $\vep>0$ be an expansivity constant for $f$.  Given $p_1, p_2\in \Per(f)$ 
we have that $h(\mathcal O_f(p_i))=\mathcal O_f(p_i)$ ($i=1,2$). Hence, diminishing $\vep>0$
if necessary, we can assume the $\vep$-neighborhood $V_i$ of the orbit $\mathcal O_f(p_i)$ ($i=1,2$) 
to satisfy 
\begin{equation}\label{eq:distH}
\dist_H ( V_1 \cup h(V_1) \cup h^{-1}(V_1) , V_2 \cup h(V_2)\cup h^{-1}(V_2) ) >0,
\end{equation}
where $\dist_H$ denotes the Hausdorff distance. Let $K(\vep) >0$ be given by the periodic gluing orbit property.

Let $p_3 \in \Per(f)$ such that $|n(p_3)| > 2 K(\vep)$; such a point does exist because $n(\cdot)$ is unbounded.
By definition of $n(\cdot)$ we get that $2 K(\vep) \le |n(p_3)| \le \frac{\pi(p_3)}2$.

Since $\vep$ is assumed to be an expansivity constant for $f$ we conclude the diameter of the dynamic ball
$$
B(p_3, k,\vep):=\{ x\in M \colon d(f^j(x), f^j(p_3))<\vep \; \text{ for every } -k \le j \le k\}
$$
tends to zero as $k\to\infty$. 
Let $\eta=\eta_{p_3}>0$ be given by Lemma~\ref{le:continuity} and $k_3\ge 1$ be so that 
$\diam(B(p_3, k_3,\vep))<\eta$.
Given $k\ge 1$ arbitrary, the periodic gluing orbit property assures the existence of a periodic point $p \in \Per(f)$ and times $0\le t_1, t_2, t_3 \le K(\vep)$
so that:
\begin{enumerate}
\item $d(f^j(p), f^j(p_1))\le \vep$ for every $0\le j \le k \pi(p_1)$ 
\item $d(f^{j+t_1+k \pi(p_1)} (p), f^j(p_2))\le \vep$ for every $0\le j \le \pi(p_2)$, 
\item $d(f^{j+t_2+\pi(p_2) + t_1+k \pi(p_1)} (p), f^j(p_2))\le \vep$ for every $0\le j \le k_3 \pi(p_3)$, and 
\item $\pi(p)=t_3+\pi(p_3)+t_2+\pi(p_2) + t_1+k \pi(p_1)$.
\end{enumerate} 
In particular, one can choose $k\gg 1$ so that $k\pi(p_1) \ge \max\{3K(\vep) +\pi(p_2)+k_3 \pi(p_3), |n(p_3)|\}$.
Since the orbit of $p$ intersects the $\eta$-neighborhood of the orbit of $p_3$ 
and $n(f^j(p_3))=n(p_3)$ for every $0\le j \le \pi(p_3)-1$,
Lemma~\ref{le:continuity} assures that $|n(p)|> K(\vep)$.
By items (1) and (2), the point $z=f^{t_1+k \pi(p_1)} (p)$  belongs to $V_2$ and $f^{-j}(z) \in V_1$ for every 
$t_1 \le j \le k\pi(p_1)$. Moreover, by the choice of $k\ge 1$, we have that $f^{-|n(z)|}(z)=f^{-|n(p)|}(z) \in V_1$.
As either $h(p)=f^{-|n(z)|}(z)$ or $h^{-1}(p)=f^{-|n(z)|}(z)$ this leads to a contradiction with \eqref{eq:distH}.
Thus $n(\cdot) : \Per(f) \to \mathbb Z$ is bounded, which proves the proposition.
\end{proof}

\section{Trivial centralizers on hyperbolic basic pieces of $C^r$-generic diffeomorphisms}\label{sec:proofs}

Using that hyperbolic basic pieces admit analytic continuations $g \mapsto \Lambda_g$, Corollaries~\ref{cor:0}
and ~\ref{cor:A} on the triviality of the centralizer of $C^r$-generic diffeomorphisms on hyperbolic basic pieces 
are consequences of Lemma~\ref{le:periodic-spectra} below.
First we establish some notation. Given a $C^1$-diffeomorphism $f$ on a compact Riemannian manifold $M$ and 
$p\in \text{Per}(f)$, let $\sigma(Df^{\pi(p)}(p)) \subset \mathbb C$ denote the spectrum of the linear transformation 
$Df^{\pi(p)}(p): T_p M \to T_p M$, where $\pi(p)\ge 1$ is the prime period of $p$.

\begin{lemma}\label{le:periodic-spectra}
Let $1\le r\le \infty$, $f_0\in \text{Diff}^{\,r}(M)$ and $\Lambda_{f_0} \subset M$ be a maximal invariant set for $f_0$. 
Let $\mathcal U \subset \text{Diff}^{\,r}(M)$ be an open neighborhood of $f_0$ and let
$U \subset M$ be open set such that the analytic continuation of the hyperbolic basic sets ~\eqref{eq:analytic}
is well defined. Then, there exists a $C^r$-residual subset $\mathcal R \subset \mathcal U$ such that for any
$f\in \mathcal R$ the following holds: $Df^{\pi(p)}(p)$ is not linearly conjugated to $Df^{\pi(q)}(q)$ for any periodic 
points $p,q$ so that $p\notin \mathcal O_f(q)$.
\end{lemma}

\begin{proof}
Fix $r\ge 1$ arbitrary. For every $n\ge 1$ it is enough to prove that the set 
\begin{align*}
\mathcal R_n :=\big\{  f\in \cU \colon &  \sigma(Df^{\pi(p)}(p)) \cap \sigma(Df^{\pi(q)}(q)) =\emptyset 
	 \;\text{for every} \; p,q \in \text{Per}_{\le n}(f), p\notin \mathcal O_f(q)  \big\}
\end{align*}
forms a $C^r$-open and dense subset of $\cU$ (hence $\mathcal R=\bigcap_{n\ge 1} \mathcal R_n$
is the desired residual subset of $\cU$).

Fix $n\ge 1$. By the analytic continuation of the hyperbolic basic sets, 
every periodic point $p\in \text{Per}(f_0)$ admits an analytic continuation 
$p_f \in \text{Per}(f)$ for every diffeomorphism $f\in \mathcal U$ and $\pi(p_f)=\pi(p)$. 
Moreover, since the spectrum $\sigma(Df^{\pi(p)}(p)) \subset \mathbb C$ is compact and the map
$f\mapsto \sigma(Df^{\pi(p)}(p_f))$ is continuous (in the Hausdorff distance) 
then it is clear that $\mathcal R_n$ is a $C^r$-open subset in $\mathcal U$.

We proceed to prove that $\mathcal R_n \subset \cU$ is $C^r$-dense. The proof consists of showing that whenever 
the spectrum of the derivative at periodic points of period smaller or equal to $n$ intersect, one can perform a finite 
number of arbitrarily small $C^r$-perturbations with disjoint supports (which resemble the ones considered by Franks~\cite[Lemma~1.1]{Franks}) and to obtain an element of $\mathcal R_n$. 
Take $f\in \cU$ and $\vep>0$ arbitrary.
Since there are finitely many periodic orbits of period smaller or equal to $n$ in $\Lambda_f$, 
choose periodic points $(p_i)_{i=1}^\ell$
in such a way that every periodic point in $\text{Per}_{\le n}(f)$ belongs to the orbit of exactly one of the points $p_i$,
 for some $1\le i \le \ell$.
Pick $\delta>0$ small so that the exponential map $\exp_{x}: B(0,\delta) \subset T_x M \to B(x,\delta)$ is a $C^\infty$ 
diffeomorphism for every $x\in M$, and so that the family of open balls $\{ B(p_i,\delta) : 1\le i \le \ell\}$ on $M$ is pairwise disjoint 
and every ball contains exactly one point of every periodic orbit of period smaller or equal to $n$.
We proceed recursively on $2\le i \le \ell$. Set $g_1=f$. If $\sigma(Dg_1^{\pi(p_1)}(p_1)) \cap \sigma(Dg_1^{\pi(p_2)}(p_2))
= \emptyset$ then take $g_2=g_1$, otherwise take
\begin{equation}
g_{2,A_2}:=\begin{cases}
	\begin{array}{cl}
	g_1(x) & , \text{ if } x\notin  B(p_2,\delta) \\
	\exp_{f(p_2)}\circ [ (1- \beta) \hat g_1 + \beta (\hat g_1 \cdot A_2)] \circ \exp_{p_2}^{-1} & , \text{ if } x\in  B(p_2,\delta) 
	\end{array}
	\end{cases}
\end{equation}
where $\hat g_1 : B(0,\delta) \to B(0,\delta)$ is given by $\hat g_1= \exp_{g_1(p_2)}^{-1} \circ \; g_1 \circ \exp_{p_2}$,
the map $\beta : B(0, 2\delta) \to [0,1]$ is a $C^\infty$ bump function so that $\beta\mid_{B(0, 2\delta) \setminus B(0,\delta)} \equiv 0$
and $\beta\mid_{B(0, \delta/2)}\equiv 1$, and $A_2$ is $C^\infty$-small perturbation of the identity map.
It is clear from the construction that $g_{2,A_2}$ is a $C^r$-diffeomorphism, coincides with $g_1$ 
outside of the ball $B(p_2,\delta)$ and  $\text{Per}_{\le n}(g_{2,A_2})=\text{Per}_{\le n}(g_1)$.
Moreover, by Fa\`{a} di Bruno's formula, given $q\in B\left(p,\delta\right)$, we have
\begin{equation}\label{FDB}
\frac{\partial^n \beta(\|q\|^2)}{\partial q_k}
	=\sum\frac{n!}{m_1!1!^{m_1}m_2!2!^{m_2}...m_n!n!^{m_n}}\beta^{(m_1+...+m_n)}(\|q\|^2)
	\prod_{j=1}^n \left(\frac{\partial^j \|q\|^2}{\partial q_k}\right)^{m_j},
\end{equation}
where the sum is over all vectors with nonnegative integers entries $(m_1,..., m_n)$ such that we have 
$\sum_{j=1}^n j.m_j=n$. These derivatives are bounded above by a constant $C$ (depending on $\delta$).
Hence $g_{2,A_2}$ can be taken $\vep$-$C^r$ arbitrarily close to $f$ provided 
the map $A_2$ is sufficiently $C^\infty$ close to the identity map. More precisely,
\begin{equation}\label{eq:P}
	\text{there exists $\zeta>0$ so that:} \;\;  \|A_2-Id\|_{C^\infty}<\zeta \Rightarrow \|g_{2,A_2} - g_1\|_{C^r}<\vep. 
\end{equation}
Furthermore, by construction 
\begin{equation}\label{eq:formula}
Dg_{2,A_2}(p_2)= D (\hat g_1\cdot A_2)(p_2).
\end{equation}
Taking $A_2=\lambda_2 Id$ we deduce that
$
\sigma(Dg_{2,A_2}(p_2)) =\{ \lambda_2  \lambda \in \mathbb C \colon \lambda \in \sigma(Dg^{\pi(p_2)}_{1}(p_2))\}.
$
Thus, there exists $\lambda_2$ arbitrarily close to one so that $g_{2,A_2}$ is $\vep$-$C^r$-close to $g_1$
and the spectrum of $Dg^{\pi(p_2)}_{1}(p_2))$ and $Dg^{\pi(p_2)}_{2}(p_2))$ lie in circles of different radius in $\mathbb C$ and, hence, cannot intersect.
Proceeding recursively for every $2<i<\ell$, we construct a $\vep$-$C^r$-perturbation $g_{i+1}:=g_{i+1,A_{i+1}}$ 
of the diffeomorphism $g_i$
(supported on $B(p_{i+1}, \delta)$) 
in such a way that $\text{Per}_{\le n}(g_{i+1,A_{i+1}})=\text{Per}_{\le n}(g_1)$
and 
$\sigma(Dg_i^{\pi(p_s)}(p_s)) \cap \sigma(Dg_{i+1}^{\pi(p_{i+1})}(p_{i+1})) = \emptyset$
for every $1\le s \le i$.
Then $g=g_{\ell}$ is a $C^r$-diffeomorphism that is $\vep$ $C^r$-close to $f$ and satisfies the requirements of the lemma.
Since $\vep$ was chosen arbitrary this proves that $\mathcal R_n$ is $C^r$-dense and finishes the proof of the lemma.
\end{proof}

\begin{remark}
Given $1\le r \le \infty$, the previous arguments also yield  a similar result among the classes of $C^r$ volume preserving 
diffeomorphisms using perturbations $A_2$ $C^\infty$ close to the identity map obtained as a finite composition of 
small rotations instead of homothetic perturbations in ~\eqref{eq:P} and \eqref{eq:formula} above
(see e.g. \cite[Proposition~7.4]{BDP} and \cite{LLS} for the harder problem of $C^1$-realization of derivatives on the conservative Franks lemma). 
\end{remark}

\subsection{Proof of Corollary~\ref{cor:A}}
Fix $1\le r \le \infty$. We will prove the corollary in the case of Anosov diffeomorphisms in 
$\cA^r(M)$ (the proof for volume preserving diffeomorphisms $\cA_m^r(M)$ is completely analogous).
Note that $M$ is a hyperbolic basic piece for any Anosov diffeomorphism in $\cA^r(M)$ 
(due to transitivity). 

Let $\mathcal D$ be a countable $C^r$-dense set of Anosov diffeomorphisms in $\cA^r(M)$. 
By structural stability, given $f\in \mathcal D$ the topological class of $f$ contains a $C^r$ open neighborhood 
$\cU_f \subset \text{Diff}^{\, r}(M)$ of $f$. By Lemma~\ref{le:periodic-spectra}, there exists a $C^r$-residual subset 
$\mathcal R_f \subset \cU_f$ of Anosov diffeomorphisms so that  $Dg^{\pi(p)}(p)$ is not linearly conjugated to 
$Dg^{\pi(q)}(q)$ for any $g\in \mathcal R_f$ and any periodic 
points $p,q \in \Per(g)$ so that $p\notin \mathcal O_g(q)$.

We claim that every $g\in \mathcal R_f$ has trivial $C^1$-centralizer.
Indeed, given $g\in \mathcal R_f$ and $h\in \mathcal Z^1(g)$ it is enough  to show that $h$ preserves periodic orbits
(cf. Theorem~\ref{thm:PGT}).
But, using $h\circ g^n = g^n\circ h$ for every $n\in\mathbb Z$, if $p$ is any periodic point of prime period $n=\pi(p)$ 
then 
$
Dg^{\pi(p)}(h(p)) = Dh ( p )\cdot \,Dg^{\pi(p)}(p) \cdot\, [Dh(p)]^{-1}.
$
This shows that $Dg^{\pi(p)}(h(p))$  and $Dg^{\pi(p)}(p)$ are linearly conjugated. 
By the construction of the residual subset $\mathcal R_f$ we conclude that 
 $h$ preserves the periodic orbits of $g$ and, consequently,
the $C^1$-centralizer of every $g\in \mathcal R_f$ is trivial.
Then, $\mathcal R:= \bigcup_{f \in \mathcal D} \mathcal R_f$ 
is a $C^r$-residual subset of $\cA^r(M)$ that satisfies the requirements of the corollary.
This proves Corollary~\ref{cor:A}.

\section{Anosov diffeomorphisms with non-trivial $C^1$-centralizer}\label{sec:example}

Here we make the construction of a linear hyperbolic automorphism satisfying the requirements
of Example~\ref{thm:C}.
Consider the three-dimensional torus $\mathbb T^3 = \mathbb R^3/ \mathbb Z^3$ and the matrix 
$A \in \text{SL}(3,\mathbb Z)$ given by 
$$
A = 
\left(
\begin{array}{ccc}
 0 & 1 & 1 \\ 
 2 & 1 & 0 \\
 1 & 0 & -1
\end{array}
\right)
$$ 
The characteristic polynomial is $p(\lambda)= -\lambda^3 + 4\lambda +1$, which has three real zeros 
$\lambda_1 < \lambda_2 < 0 < \lambda_3$ with $|\lambda_1|, |\lambda_3|>1$ and $|\lambda_2|<1$.
Indeed, the eigenvalues of $A$ are $\lambda_1 \in \, ]-1.87,-1.86[$, $\lambda_2 \in \, ]-0.26, -0.25[$ and 
$\lambda_3 \in \, ] 2.11,2.12[$. 
 Thus $A$
induces a linear Anosov diffeomorphism $f_A$ on $\mathbb T^3$ with a $Df_A$-invariant splitting
$T \mathbb T^3 = E^s_A \oplus E^u_A$ with $\dim E_A^s=1$ and $\dim E_A^u=2$.
Since $f_A$ is Anosov, hence expansive, it follows that the centralizer $Z^0(f_A)$ is discrete \cite{Wa70} (hence
$Z^1(f_A)$ is discrete). We proceed to
analyze the triviality of the $C^1$-centralizer.

It is not hard to check from algebraic manipulations, that we omit, that
$B\in \text{SL}(3,\mathbb Z)$ is so that $AB=BA$ if and only if $B$ is of the form
$$
B = 
\left(
\begin{array}{ccc}
c_1 + 2 c_2 + c_3 & c_1 + 2 c_2 &  c_1 \\ 
2 c_1 + 4 c_2  & 2 c_1 + 3 c_2 + c_3 & 2c_2  \\
c_1  & c_2 & c_3
\end{array}
\right)
$$ 
for $c_1, c_2, c_3\in \mathbb Z$ under the restriction that $|\det B|=1$.
In particular, taking $c_1=c_3=1$ and $c_2=0$ we obtain the matrix
$$
B = 
\left(
\begin{array}{ccc}
2 & 1 & 1   \\ 
2 & 3 & 0 \\
1 & 0 & 1 
\end{array}
\right)
\in SL(3, \mathbb Z)
$$ 
which is hyperbolic and commutes with $A$. 
The matrix $B$ has eigenvalues $0<\mu_1 <1 <\mu_2< \mu_3$, 
thus there exists a $Df_B$-invariant splitting
$T \mathbb T^3 = E^s_B \oplus E^u_B$ with $\dim E_B^s=1$ and $\dim E_B^u=2$.
By direct computations we get that the following three quocients (relative to eigenvalues along the 
same one-dimensional eigenspaces)
$$
\frac{\log |\mu_1| }{\log |\lambda_1| } \approx -3.26
\quad
\frac{\log |\mu_2| }{\log |\lambda_2| } \approx -0.41
\quad\text{and}\quad
\frac{\log |\mu_3| }{\log |\lambda_3| } \approx 1.89
$$
are pairwise distinct. Since $A$ and $B$ have the same one-dimensional eigenspaces, 
there are no non-zero integers $m,n$ such that $A^n=B^m$.

\section{Centralizers of Anosov diffeomorphisms and positive entropy elements}

\subsection{Proof of Theorem~\ref{thm:measures}}
Assume that the homeomorphism $f \in \Homeo$ has  finite topological entropy and that the entropy map 
$\mu \mapsto h_\mu(f)$ is upper semicontinuous: given invariant measures so that 
$\lim_{n\to\infty} \mu_n = \mu$ (in the weak$^*$ topology) it holds that $\limsup_{n\to\infty} h_{\mu_n}(f) \le h_\mu(f)$.
Under these assumptions $f$ has at least one equilibrium state for every continuous potential $\phi$ (see e.g. \cite{Wa}).

Fix an arbitrary $g\in \mathcal Z^0(f)$.  As $g\circ f = f\circ g$ then $g_*\circ f_* = f_*\circ g_*$ 
where the push-forward dynamics $f_* : \cM(M) \to \cM(M)$ is given by $(f_*\eta)(A)= \eta ( f^{-1} A)$ for every 
Borelian $A$ in $M$ (and $g_*$ is defined similarly). 
Let $\mu$ be a maximal entropy measure for $f$.
It is not hard to check from the commutativity of $f$ and $g$ that
$g_* \mu$ is an $f$-invariant probability measure and 
that  $h_{g_* \mu} (f) =h_{\mu} (f)$.  The first is a direct consequence of the fact that $g_*\circ f_* = f_*\circ g_*$.
The second follows as, for every finite partition $\cP$ on $M$, 
\begin{align*}
h_{g_*\mu} (f, \cP)
	 & = \inf_{n\ge 1} \frac1n H_{g_* \mu} \big( \bigvee_{\ell=0}^{n-1} f^{-\ell}(\cP) \big) 
	 = \inf_{n\ge 1} \frac1n H_{\mu} \big( g^{-1} \big(\bigvee_{\ell=0}^{n-1} f^{-\ell}(\cP) \big) \big)\\
	& = \inf_{n\ge 1} \frac1n H_{\mu} \big( \bigvee_{\ell=0}^{n-1} f^{-\ell}(g^{-1} (\cP)) \big) 
	 = h_{\mu} (\, f, g^{-j} (\cP) \,).
\end{align*}
Since the partitions $\cP$ (resp. $g^{-j} (\cP)$) can be taken with arbitrarily small diameter
then we conclude that $h_{g_* \mu} (f) =h_{\mu} (f)$, as claimed. 
Recursively,  $h_{g_*^j \mu} (f) =h_{\mu} (f)$ for every $j\ge 0$.
As the measure theoretic entropy function is affine we get that  
the $f$-invariant probability measures 
\begin{equation}\label{eq:Cesaro}
\mu_n := \frac1n \sum_{j=0}^{n-1} g_*^j \mu
\end{equation}
lie in the closed simplex $\mathcal S :=\{ \eta \in \cM_f(M) \colon h_{\eta}(f)= h_{\mu}(f)\}$
for every $n\ge 1$.
Since both $f_*$ and $g_*$ are continuous, if $\mu_\infty$ is an accumulation point of
the sequence $(\mu_n)_n$ then 
${\mu_\infty} \in \cM_f(M) \cap \cM_1(g)$. Moreover,
$
h_{\mu_\infty}(f) \ge \limsup_{n\to\infty} h_{\mu_n}(f) = h_{\mu}(f)
$
by the upper semicontinuity of the entropy map $\cM_f(M) \ni \eta \mapsto h_\eta(f)$.
In consequence $\mu_\infty$ is a maximal entropy measure for $f$ that is preserved by both 
$f$ and $g$. In consequence $h_{\text{top}}(f) = \sup_{\mu \in \cM_f(M) \cap \cM_1(g)} h_\mu(f)
=h_{\mu_\infty}(f).$
This proves the first statement in the theorem.

Finally, if $\mathcal Z^0(f)$ is finitely generated and $G=\{g_1, \dots, g_k\}$ is a set of generators 
then the previous argument shows that every push-forward $(g_{i+1})_*$ preserves the simplex 
$\mathcal S_+=\{ \eta \in \cM_f(M) \colon h_{\eta}(f)= \htop(f)\}$ for every $1\le i \le k$. Then, if
$\mu$ is a maximal entropy measure for $f$ then any accumulation point of the probability measures
$$
\frac1{n^k} \sum_{j_1=0}^{n-1} \dots \sum_{j_k=0}^{n-1}  (g_1)_*^{j_1} \dots  (g_k)_*^{j_k} \mu
$$
is a maximal entropy measure for $f$ that is preserved by all elements in $\mathcal Z^0(f)$. This finishes the
proof of the theorem.

\begin{remark}
If $g\in \mathcal Z^0(f)$ and $\mathcal E_\phi \subset \cM_f(M)$ denotes the space of equilibrium states 
for $f$ with respect to a continuous potential $\phi$, 
similar computations as above yield that any accumulation point $\mu_\infty$ of the sequence of 
invariant probability measures defined by the Cesaro averages~\eqref{eq:Cesaro} (taking 
$\mu\in \cE_\phi$ instead of a maximal entropy measure) satisfies 
$
h_{\mu_\infty}(f) + \int \phi \,d\mu_\infty
	\ge  h_{\mu}(f) + \int \phi \,d\mu_\infty,
$
where
$
\int \phi \,d\mu_\infty 
	=\lim_{k\to\infty} \frac1{n_k} \sum_{j=0}^{n_k-1} \int \phi \circ g^j \,d\mu
$
for some subsequence $(n_k)_k$. Therefore, one can only expect $g_*(\cE_\phi) \subset \cE_\phi$ in the very
special case that $g_*$ preserves the level set $\{ \eta \in \cM_f(M) \colon \int \phi \, d\eta= \int \phi \, d\mu\}$.
This is true in the case that $\phi$ is constant, when the equilibrium states are maximal entropy measures.
\end{remark}

\subsection{Proof of Proposition~\ref{prop:rigiditydim2}}

Assume that $f_A : \mathbb T^2  \to \mathbb T^2$ is a linear Anosov diffeomorphism and $g\in \mathcal Z^0(f_A)$ has positive entropy. We proof makes use of the crucial fact: given the hyperbolic automorphism $f_A$ on $\mathbb T^2$ every 
element in the $C^0$-centralizer of $f_A$ that fixes the origin is linear \cite[page 100]{PY89b}. 
Therefore, for any $g\in \mathcal Z^0(f_A)$ there exists $k\in \mathbb N$ and $B\in GL(2, \mathbb Z)$ with $|\det B|=1$
so that $g^k=f_B$. 
As $f_B$ is $C^\infty$ then Pesin formula implies 
$
h_{top}(f_B)= \sum_i \lambda_i^+(f_B)
 $ 
where $\lambda_i^+(f_B)$ denote the positive Lyapunov exponents of $f_B$ with respect to Lebesgue.
Thus $h_{top}(f_B)= k \, h_{top}(g)>0$ if and only if  there exists at least one eigenvalue of $B$ with absolute value larger than  one. Since $f_B$ is volume preserving the later implies that $B\in SL(2,\mathbb Z)$ is a hyperbolic
matrix and, equivalently, $f_B$ is an Anosov automorphism.
This proves the proposition.

\begin{remark}
The centralizer of linear Anosov diffeomorphisms on $\mathbb T^2$ can indeed contain roots
(e.g. $f_B \in Z(f_A)$ where $A, B \in SL(2,\mathbb Z)$ are given by
$$
A = 
\left(
\begin{array}{cc}
2 & 1 \\
1 & 1 
\end{array}
\right)
	\quad\text{and}\quad
B = 
\left(
\begin{array}{cc}
1 & 1 \\
1 & 0 
\end{array}
\right)
$$
and verify $B^2=A$). Nevertheless, the existence of roots in the centralizer is rare for more regular diffeomorphisms. 
Indeed, for every $k\in \mathbb N \cup\{\infty\}\cup\{\omega\}$ there exists an $C^k$-open and 
dense set of positive entropy real analytic diffeomorphisms on surfaces that have trivial centralizer
\cite{Ro94}.
\end{remark}

\section{Further questions:}

Given a $C^r$-diffeomorphism $f\in \text{Diff}^r(M)$ and $0\le k \le r$, the centralizer $Z^k(f)$ is a 
subgroup of $\text{Diff}^k(M)$. In the case there exists
an open subset $\cU \subset \text{Diff}^{\,r}(M)$ so that $Z^k(f)$ is finitely generated for every $f\in \cU$
it makes sense to study the continuity points of the map 
$
\cU \ni f \mapsto N^k(f)
$
 where $N^k(f) \in \mathbb N \cup\{\infty\}$ denotes the minimum number of generators for $Z^k_0(f)$.  
Some results can be deduced in the case $r=k=1$.
Since $C^1$-generically in $\text{Diff}^{\,1}(M)$ the centralizer is trivial (cf. \cite{BCW})  there exists
a residual subset $\mathcal R \subset \text{Diff}^{\,1}(M)$ so that  $\mathcal R \ni f \mapsto N^1(f)$ is constant 
and equal to one.
By \cite{BF}, there exists an open subset of $C^1$-Anosov diffeomorphisms with trivial centralizer, hence 
formed by continuity points for the function $f \mapsto N^1(f)$. 
The following questions remain open:

\medskip
\noindent \emph{1. Given $r\ge 1$, the set of $C^r$-Anosov diffeomorphisms (resp. $C^r$ Axiom A diffeomorphisms 
with the no cycles condition) that have trivial centralizer forms a $C^r$ open and dense set on the space of 
the space of $C^r$-Anosov diffeomorphisms (resp. $C^r$ Axiom A diffeomorphisms with the no cycles condition)?}

\medskip

\noindent \emph{2. If $1\le k \le r$, what are the continuity points of the functions $N^k$? Are all Anosov diffeomorphisms
points of (semi)continuity? Are there Anosov diffeomorphisms whose centralizers are not finitely generated?}
\medskip

Actually not all Anosov diffeomorphisms are continuity points of this function, even if we restrict to linear hyperbolic
diffeomorphisms.  It is definitely interesting to discuss the triviality of the centralizer with the set of entropies associated to each element  of the centralizer. Indeed, observe that the triviality of the centralizer $Z^0(f)$ seldom implies that
\begin{equation}\label{eq:entropyset}
H_f :=\{ \htop(g) : g \in \mathcal Z^0(f)\}
\end{equation}
coincides with the arithmetic progression $\{ n \htop(f) : n\in \mathbb N_0\}$. However, it is not true that this condition 
implies on the triviality of the centralizer, as it can be easily observed in the examples constructed by the first author
in~\cite{Ro05}. 
The following question is suggested by this class of examples:
\medskip

\noindent \emph{3. Assume that $f : \mathbb T^n \to \mathbb T^n$ is a $C^1$ Anosov diffeomorphism 
and that $H_f=\{ n \htop(f) : n\in \mathbb N_0\}$. Is it true that all elements in the centralizer are either
of the form $f^k$, $k\in \mathbb Z$, or roots of the identity?}
\medskip

Finally, we believe that positive entropy elements in the centralizer of an Anosov diffeomorphism
should have some rigidity condition. So, we ask:

\medskip
\noindent \emph{4. Let $f$ be an Anosov diffeomorphism on $\mathbb T^n$ $(n\ge 2$). Are all positive entropy elements in the $C^1$-centralizer of $f$ partially hyperbolic?} 
\medskip

\medskip
\subsection*{Acknowledgements} 
The first author was partially supported by by CMUP \- (UID/MAT/ 00144/2013), which is funded by FCT (Portugal) with national (MEC) and European structural funds through the programs FEDER, under the partnership agreement PT2020
and by PTDC/MAT-CAL/3884/2014.
The second author was partially supported by BREUDS and CNPq-Brazil. This work was developed during the visit 
of the second author to University of Porto, whose hospitality and research conditions are greatly acknowledged. 
The authors are grateful to W. Bonomo and P. Teixeira for useful comments.

\end{document}